\documentclass[a4paper, 11pt]{amsart}

\usepackage[T2A,T1]{fontenc}  
\usepackage[utf8]{inputenc}

\usepackage[russian, french, british]{babel}

\usepackage{amssymb, amsmath}
\usepackage{mathtools}
\usepackage{bm}

\usepackage[cal=boondoxo]{mathalfa}

\usepackage{comment}
\usepackage{url}

\usepackage[hidelinks]{hyperref}

\newtheorem{THEOREM}{Theorem}[section]

\newtheorem{theorem}[THEOREM]{Theorem}
\newtheorem{lemma}[THEOREM]{Lemma}
\theoremstyle{definition}
\newtheorem{definition}[THEOREM]{Definition}

\newtheorem{fact}[THEOREM]{Fact}
\newtheorem{remark}[THEOREM]{Remark}
\newtheorem{Observation}[THEOREM]{Observation}
\newenvironment{observation}{\begin{Observation}}
{\end{Observation}}
\newtheorem{Notation}[THEOREM]{Notation}
\newenvironment{notation}{\begin{Notation}}
{\end{Notation}}

\newcommand{\forces}{\Vdash}

\newcommand{\supt}{\text{supt}}

\newcount\skewfactor
\def\mathunderaccent#1#2 {\let\theaccent#1\skewfactor#2
\mathpalette\putaccentunder}
\def\putaccentunder#1#2{\oalign{$#1#2$\crcr\hidewidth
\vbox to.2ex{\hbox{$#1\skew\skewfactor\theaccent{}$}\vss}\hidewidth}}
\def\name{\mathunderaccent\tilde-3 }

\newcommand{\rest}{\upharpoonright}

\newcommand{\cf}{\text{cf}}

\newcommand{\eop}{\bigstar}
\newcommand{\DD}{\mathcal D}

\newcommand{\eod}{\scriptstyle\blacksquare}

\title{A note on iterating strongly $(<\lambda)$-closed  stationary $\lambda^+$-cc forcing}

\author{Mirna D{\v z}amonja, by-name Logique Consult, Paris}

\date{}

\begin{document}

\maketitle

\begin{abstract} We give an exposition of an iteration theorem for iterating $(<\lambda)$-closed  stationary $\lambda^+$-cc forcing with supports of size $<\lambda$
and preserving these two properties. We discuss the relation of this theorem with other iteration theorems and forcing axioms that have appeared in the literature, notably the one from \cite{Sh80}.
\end{abstract}

\section{Introduction}\label{sec:intro}  In his work on generalising Martin's axiom to higher cardinals \cite{Sh80}, 
Saharon Shelah introduced the notion of stationary $\lambda^+$-cc. Various versions of forcing axioms using this notion have appeared, including our joint work \cite{5authorsforcing}. This note aims to give a clean presentation of the main arguments and clarify some misstatements surrounding the topic.

Theorem \ref{th:iteration} is our rendition of a theorem due to Shelah \cite[Th 1.1]{Sh80}. Shelah's theorem is stated in terms of $(<\lambda)$-closed forcing rather than strongly $(<\lambda)$-closed forcing (see Definition \ref{def:strongclosure}). In Remark \ref{rem:whystrong} we comment on the difference between the two notions and why that difference does not play a role in the applications.

On the other hand, Shelah also has the requirement that any two compatible conditions have the least upper bound (this is called {\em well met}) and a  requirement on the size of the individual forcing. The latter is likely done in view of preserving cardinals
$\lambda^{++}$ and above and a wish to have a forcing axiom rather than an iteration theorem. The proof provided in \cite{Sh80} is written in terms which are not precise enough, so we are not entirely sure that his theorem as stated is correct, although we have used it and cited it often in our own work, including \cite{DzShPrikry}, \cite{withCharlesandKomjath}, \cite{3authorsforcing}. 
This was done having in mind our rendition of the theorem as a back up, considering that it was just a minor improvement of what Shelah had in mind in
\cite[Th 1.1]{Sh80}. In all the applications cited, the theorem as presented here applies. We describe some of the other relevant papers and theorems that have appeared.

A strengthening of Shelah's theorem where the condition of every two conditions having the least upper bound was replaced by the condition called {\em countable 
parallel closure} in \cite[Th. 1.2]{5authorsforcing}. The reason this was done is that in the application in the paper, the individual forcing used is not well met. Perhaps due to differences in the approach of doing forcing, this author is not entirely convinced by the proof we presented
in \cite[Th. 1.2]{5authorsforcing}. This does not alter the truth of the main theorems in \cite{5authorsforcing} since the applications of \cite[Th. 1.2]{5authorsforcing} can simply be replaced by applications of Theorem \ref{th:iteration}, given that the individual forcings in question fit the required framework.

The author had been curious about what Shelah really had in mind in his theorem \cite[Th 1.1]{Sh80}, so at the occasion of a preprint by John Baldwin, Alexei Kolesnikov and Shelah proving a theorem by an application of \cite[Th 1.1]{Sh80} without the well-met condition, she pointed out to the authors that they are misquoting
\cite[Th 1.1]{Sh80} (see the thanks at the end of the Introduction of that paper). This resulted in them improving the preprint to include in the published version of the paper a discussion and a proposition that the well-met condition is after all not necessary \cite[Rem. 3.5 and Prop. 3.6]{zbMATH05609396}, with a short proof different than what is provided here. We are not sure to understand their proof, but this may be due to the fact that we have not put much effort in studying it, since we only saw it in its published version several years after the discussions were taking place.  The discussion also resulted in the paper
\cite{zbMATH07736866} by Shelah in which he concludes that, after all, some form of the well met condition is necessary in the numerous versions of his forcing axioms derived from the one in \cite[Th 1.1]{Sh80}. 

It would therefore appear that Theorem \ref{th:iteration} finally is not the theorem that Shelah originally had in mind. His verdict on his two papers
\cite{zbMATH05609396} and \cite{zbMATH07736866} which come to contradictory conclusions is not clear. 

We note that our Theorem \ref{th:iteration} is not stated as a forcing axiom but as an iteration theorem. Having a forcing axiom out of an iteration theorem is
considered straightforward, but in our opinion one has to be careful as to the fact that there might be forcing notions in the extension that fit the
framework of a forcing axiom and yet were not considered during the iteration. 

A brief look at the applications
intended in \cite{zbMATH05609396}, which are \cite[Cl. 3.7]{zbMATH07736866} (with the $(<\mu)$-closure as proved) and 
\cite[Cl. 3.14]{zbMATH07736866}  
(where the proof of the $(<\mu)$-closure is omitted), would seem to imply that Theorem \ref{th:iteration} might be sufficient. This is as the
$(<\mu)$-closure of the forcing seems to us to be witnessed in a way sufficient for an application of Theorem \ref{th:iteration}. 

In conclusion, various papers appearing in the years since  \cite{Sh80} have cast a doubt in our mind at our assumption that it is obvious that Shelah must have had in mind something as Theorem \ref{th:iteration}. Hence, we have felt that it is of interest to present the theorem directly as is done in this note.

\section{Definitions and Basic Facts}\label{sec:def} Throughout, $\lambda>\aleph_0$ is a cardinal satisfying $\lambda^{<\lambda}=\lambda$.

\begin{notation}\label{not:s-lambda+-lambda} Suppose that $\theta$ is a regular cardinal and $\alpha>\theta$. We denote
\[
S^\alpha_\theta=\{i<\alpha:\cf(i)=\theta\}.
\]
\end{notation}

\begin{definition}\label{def:statt-cc} A forcing notion $\mathbb P$ is {\em stationary $\lambda^+$-cc with respect to a 
sequence $\bar{p}=\langle p_\alpha:\,\alpha<\lambda^{+}\rangle$ of conditions in $\mathbb P$} if there is a club $E=E_{\bar{p}}\subseteq \lambda^+$ and a regressive function $f=f_{\bar{p}}$ on $E$ such that for every $\alpha, \beta\in S^{\lambda^+}_\lambda$ we have 
\[
f(\alpha)=f(\beta)\implies p_\alpha, p_\beta \mbox{ are compatible}.
\]
We say that $\mathbb P$ is {\em stationary $\lambda^+$-cc} if it is stationary $\lambda^+$-cc with respect to 
$\bar{p}$ for every sequence $\bar{p}=\langle p_\alpha:\,\alpha<\lambda^{+}\rangle$ of conditions in $\mathbb P$.
$\eod_{\mbox{\tiny{Def.}} \ref{def:statt-cc}}$
\end{definition}

\begin{observation}\label{obs:implies-cc} A stationary $\lambda^+$-cc is in particular $\lambda^+$-cc.
\end{observation}

\begin{proof} Suppose that $\mathbb P$ is a stationary $\lambda^+$-cc forcing. Let $\bar{p}=\langle p_\alpha:\,\alpha<\lambda^{+}\rangle$ 
be a sequence of conditions in $\mathbb P$ and let $E=E_{\bar{p}}$ and $f=f_{\bar{p}}$ be as in Definition \ref{def:statt-cc}. Let 
$S\subseteq S^{\lambda^+}_\lambda\cap E$ be stationary such that $f$ is constant on $S$. Hence for every $\alpha, \beta \in S$ the conditions
$p_\alpha, p_\beta$ are compatible. In particular $\bar{p}$ does not enumerate an antichain.  
$\eop_{\ref{obs:implies-cc}}$
\end{proof}

\begin{fact}\label{fact:clubs} Suppose that $\mathbb P$ is $\lambda^+$-cc and $\name{D}$ is a $\mathbb P$-name of a club of $\lambda^+$. Then there is
a club $C$ of $\lambda^+$ such that $\mathbb P$ forces that $C\subseteq \name{D}$.
\end{fact}

\begin{proof} Let $\name{f}$ be a $\mathbb P$-name for an increasing enumeration of $\name{D}$. By the property of $\lambda^+$-cc, there is
a function $F\in \mathbf V$ associating to every ordinal $\alpha<\lambda^+$ a set of $\le\lambda$ many ordinals in $\lambda^+$ such that
it is forced by $\mathbb P$ that for every $\alpha <\lambda^+$ we have $\name{f}(\alpha)\in F(\alpha)$. Then there is a club $C$ of $\lambda^+$
satisfying that for every $\delta\in C$ and $\alpha<\delta$ we have $\sup(F(\alpha))<\delta$. It follows from the fact that 
$\name{D}$ is forced to be a club that for every such $\delta$ we must have $\name{f}(\delta)=\delta$ and in particular $\mathbb P$ forces that $C\subseteq \name{D}$.
$\eop_{\ref{fact:clubs}}$
\end{proof}

\begin{fact}\label{lem:combining} Suppose that $\langle f_i:\,i<\theta\rangle$ are regressive functions on some club $C$ of $\lambda^{+}$, for some $1<\theta<\lambda$.
Then there is a club $D$ of $\lambda^{+}$ and a regressive function $g$ on $D$ such that for every $\alpha, \beta\in C\cap D\cap S^{\lambda^+}_\lambda$ we have
\[
[g(\alpha)=g(\beta)] \implies \bigwedge_{i<\theta}  [f_i(\alpha)= f_i(\beta)].
\]
\end{fact}

\begin{proof} Let $H:{}^{\theta} \lambda^{+}\to \lambda^{+}$ be a bijection and let $D$ be a club of $\lambda^{+}$ such that for 
$\delta\in D\cap S^{\lambda^+}_\lambda$ we have
\[
(\forall \langle \beta_i:\;i<\theta\rangle \in {}^\theta\delta) [H ( \langle \beta_i:\;i<\theta\rangle)<\delta].
\]
For $\delta\in D$ let $g(\delta)=H ( \langle f_i(\delta):\;i<\theta\rangle)$. By the definition of $D$ we have that $g$ is regressive on $D$ and by the definition of $H$ we have that the required equation is satisfied.
$\eop_{\ref{lem:combining}}$
\end{proof}

\begin{definition}\label{def:strongclosure} A {\em strongly $(<\lambda)$-closed forcing notion} is a pair $({\mathbb P},F)$ where ${\mathbb P}$
is a forcing and
$F$ is a function which acts on increasing sequences of length $<\lambda$ of conditions in ${\mathbb P}$, so that for every 
such sequence $\bar{p}=\langle p_i:\,i<i^\ast\rangle$ where $i^\ast<\lambda$ we have that $F(\bar{p})$ is a common upper bound to $\bar{p}$.

In such circumstances we say that $F$ {\em witnesses} the strong $(<\lambda)$-closure of ${\mathbb P}$. If no confusion arises, we often just say that
${\mathbb P}$ strongly $(<\lambda)$-closed when we mean that we have fixed a function $F$ such that $({\mathbb P},F)$ is a strongly $(<\lambda)$-closed forcing notion.
$\eod_{\mbox{\tiny{Def.}} \ref{def:strongclosure}}$
\end{definition}

\begin{remark}\label{rem:whystrong} To clarify the difference between the strong  $(<\lambda)$-closure and the ordinary $(<\lambda)$-closure we must think in terms of where the function $F$ exists. In fact, the difference is visible only when we discuss the notions ``${\mathbb P}$ forces that $\name{Q}$ is a
$(<\lambda)$-closed forcing'' and ``${\mathbb P}$ forces that $\name{Q}$ is a strongly
$(<\lambda)$-closed forcing''. If ${\mathbb P}$ forces that $\name{Q}$ is a
$(<\lambda)$-closed forcing, then for every given sequence $\bar{q}=\langle \name{q}_i:\,i < i^\ast\rangle$ of names for $i^\ast<\lambda$ conditions in 
$\name{Q}$ and a condition $p$ which forces that the sequence $\langle{q}$ is increasing in $\name{Q}$, we can find a ${\mathbb P}$-name that $p$ forces to be a common upper bound of $\bar{q}$. But we do not necessarily have access to a function which gives us such a name uniformly. That is what the strong $(<\lambda)$-closure provides. 

In any application of forcing axioms or an iteration of concrete forcing, we necessarily work with strong  $(<\lambda)$-closure in which upper bounds of sequences like $\bar{q}$ are given functionally. So, although the two notions of $(<\lambda)$-closure are theoretically different from the point of view of proving iteration theorems, they are the same from the point of view of applications.
\end{remark}

\subsection{The two step iteration}\label{sec:twostep}

\begin{lemma}\label{lem:two-step} Suppose that $\mathbb P$ is a stationary $\lambda^+$-cc forcing notion and $\name{Q}$ is a $\mathbb P$-name for a stationary $\lambda^+$-cc forcing. Let $\bar{p}=\langle (p_\alpha, \name{q}_\alpha):\,\alpha<\lambda^+ \rangle$ be a sequence
of conditions in $\mathbb P\ast \name{Q}$ and let $\name{D}, \name{g}$ be $\mathbb P$-names witnessing that $\mathbb P$  forces that $\name{Q}$ satisfies the stationary $\lambda^+$-cc for the sequence $\bar{q}=\langle \name{q}_\alpha:\,\alpha<\lambda^+\rangle$.

Then the two step iteration $\mathbb P\ast \name{Q}$ satisfies:
\begin{enumerate}
\item  for every $\alpha<\lambda^+$, the set 
\[
\DD_\alpha=\{(p, \name{q})\in \mathbb P\ast \name{Q}:\,p \mbox{ decides the value of }\name{g}(\alpha) \mbox{ if defined}\}
\]
is
dense in $\mathbb P\ast \name{Q}$. Moreover, for every $(p, \name{q})\in \mathbb P\ast \name{Q}$, there is $r\in  \mathbb P$ such that 
$(r, \name{q})\in \DD$,
\item $\mathbb P\ast \name{Q}$  is stationary $\lambda^+$-cc.
\end{enumerate}
\end{lemma}

\begin{proof} (1) This follows from the definition of a two-step iteration of forcing.

\smallskip

{\noindent (2)} Given a sequence $\bar{p}=\langle (p_\alpha, \name{q}_\alpha):\,\alpha<\lambda^+ \rangle$
of conditions in $\mathbb P\ast \name{Q}$, as in the statement of the Lemma. 
Let $E,f$ be a club and a regressive function witnessing that $\mathbb P$ is a stationary $\lambda^+$-cc for $\langle p_\alpha:\,\alpha<\lambda^+\rangle$
and let 
$\name{D}, \name{g}$ be $\mathbb P$-names witnessing that $\mathbb P$  forces that $\name{Q}$ is stationary $\lambda^+$-cc
for $\langle \name{q}_\alpha:\,\alpha<\lambda^+\rangle$.
By Observation \ref{obs:implies-cc},
the stationary 
$\lambda^+$-cc of $\mathbb P$ implies that $\mathbb P$ is $\lambda^+$-cc. By Fact \ref{fact:clubs} we can
find a club $C$ in $\mathbf V$ which is forced to be included in $\name{D}$. 

By item (1), for every $\alpha
\in C$ we can extend $p_\alpha$ to a condition $r_\alpha$ such that $r_\alpha$ decides the value $h(\alpha)$ of $\name{g}(\alpha)$. Hence $h$
is a function in $\mathbf V$ and for every $\alpha\in C$ we have that $h(\alpha)<\alpha$ (since it is forced by $\mathbb P$ that $\name{g}$ is regressive
on $\name{D}$). By Fact \ref{lem:combining}, there is a function $f^\ast$ regressive on $C\cap E$ such that 
for every $\alpha, \beta\in C\cap E \cap S^{\lambda^+}_\lambda$ we have
\[
[f^\ast(\alpha)=f^\ast(\beta)] \implies [f(\alpha)= f (\beta) \wedge h(\alpha)=h(\beta)].
\]
In particular, for such $\alpha, \beta$ we have that $r_\alpha, r_\beta$ are compatible. Say $r\ge r_\alpha, r_\beta$. Then $r\forces_{\mathbb P}``\name{g}(\alpha)=\name{g}(\beta)"$ and hence $r$ forces that there is a condition $\ge \name{q}_\alpha, \name{q}_\beta$.
By the Existential Completeness Lemma of forcing it follows that there is a name $\name{q}$ such that 
$r\forces ``\name{q}\ge \name{q}_\alpha, \name{q}_\beta"$. Therefore $(r, \name{q})$ is a common upper bound to $(r_\alpha, \name{q}_\alpha)$
and $(r_\beta, \name{q}_\beta)$, so also an upper bound to $(p_\alpha, \name{q}_\alpha)$ and $(p_\beta, \name{q}_\beta)$. This shows that
$E\cap C$ and $f^\ast$ witness that $\mathbb P\ast \name{Q}$ is stationary $\lambda^+$-cc for $\bar{p}$.
$\eop_{\ref{lem:two-step}}$
\end{proof}

\subsection{The general iteration}\label{sec:anystep}
We first isolate the following Lemma \ref{lem:strong-closure}, which is part of the General Iteration Theorem \cite[Th. 7.11]{Property-B}. We
repeat the proof for convenience.

\begin{lemma}\label{lem:strong-closure} Let $\mu\ge\lambda$
be a regular cardinal. Suppose that $\langle ({\mathbb P}_\zeta, 
\name{Q}_\zeta):\,\zeta<\zeta^*\rangle$ is a $(<\mu)$-support iteration of individual forcings such that for every $\zeta$ we have that 
${\mathbb P}_\zeta$ forces that $\name{Q}_\zeta$ is 
strongly $(<\lambda)$-closed. 

Then the limit ${\mathbb P}_{\zeta^\ast}$ of the iteration is
strongly $(<\lambda)$-closed, as witnessed by a function $h_{\zeta^\ast}$, such that for every $i^\ast<\lambda$ and an increasing
sequence $\bar{p}=\langle p_i:\,i<i^*\rangle$ in ${\mathbb P}_{\zeta^\ast}$, we have
\[
\supt(h_{\zeta^\ast}(\bar{p}))=\bigcup_{i<i^\ast}\supt(p_i).
\]
\end{lemma}

\begin{proof} For every $\zeta<\zeta^\ast$, let $\name{f}_\zeta$ be a ${\mathbb P}_\zeta$-name such that ${\mathbb P}_\zeta$
forces that $(\name{Q}_\zeta, \name{f}_\zeta)$ is a strongly $(<\lambda)$-closed forcing notion.
The proof is by induction on ${\zeta^\ast}$, where we prove that ${\mathbb P}_{\zeta^\ast}$ is strongly $(<\lambda)$-closed as witnessed by a function $h_{\zeta^*}$ satisfying the requirement in the displayed equation. 

Suppose that $\langle p_i:\,i<i^*\rangle$ is an increasing sequence in ${\mathbb P}_{\zeta^\ast}$ and $i^\ast<\lambda$. By induction on 
$\zeta<\zeta^\ast$ we choose an upper bound $q_\zeta=h_\zeta(\langle p_i\rest\zeta:\,i<i^*\rangle)$, and we claim that we can do it so that 
for every $\xi<\zeta$ we have $q_\xi= q_\zeta\rest\xi$.

The case \underline{$\zeta=0$} is trivial. For the case \underline{$\zeta=\xi+1$}, this is like the proof that the iteration of two $(<\lambda)$-closed forcings is $(<\lambda)$-closed, but we need to make that property functional. By the induction hypothesis, ${\mathbb P}_\xi$ is a strongly 
$(<\lambda)$-closed forcing, as witnessed by a function $h_\xi$ satisfying the requirement of the Lemma.

Hence we may define $q_\xi=h_\xi(\langle p_i\rest\xi:\,i<i^*\rangle)$, since the
sequence $\langle p_i\rest\xi:\,i<i^*\rangle$ is increasing. Then $q_\xi$ is a common upper bound for $\langle p_i\rest\xi:\,i<i^*\rangle$ satisfying
$\supt(q_\xi)=\bigcup_{i<i^\ast}\supt(p_i\rest\xi)$.

If $\xi\notin \bigcup_{i<i^\ast}\supt(p_i)$, we define $q_\zeta=q_\xi\!\frown\!\emptyset_{\name{Q}_\xi}$ and we note that the requirements of the induction are preserved. Suppose otherwise. We then note that $q_\xi$ forces that 
$\langle p_i(\xi):\,i<i^\ast\rangle$ is increasing in $\name{Q}_\xi$ (since for each $i$ we have $p_{i+1}\rest\xi\forces_{{\mathbb P}_\xi}`` p_i(\xi)
\le_{\name{Q}_\xi} p_{i+1}(\xi)"$). Hence we have that $q_\xi$ forces that $\name{q}'_\xi$ defined by $\name{f}_\xi(\langle p_i(\xi):\,i<i^\ast\rangle)$ is a common upper bound to 
$\langle p_i(\xi):\,i<i^\ast\rangle$. 
Then $q_\zeta=(q_\xi, \name{q}'_\xi)$ is an upper bound for $\langle p_i\rest\zeta:\,i<i^*\rangle$ in ${\mathbb P}_\zeta$ and $q_\xi= q_\zeta\rest\xi$. 
Note the importance of the strong closure here which has allowed us to define $q_\zeta\rest \xi$ without changing $q_\xi$.

In conclusion, we may define 
\[
h_\zeta( \langle p_i:\,i<i^*\rangle)=
(h_\xi( \langle p_i\rest\xi:\,i<i^*\rangle), \name{f}_\xi (\langle p_i(\xi):\,i<i^\ast\rangle)).
\]
We note that we have preserved the requirements of the induction.

The case \underline{$\zeta>0$ is a limit ordinal of cofinality $<\mu$}. We intend to define $q_\zeta$ as the unique condition in ${\mathbb P}_\zeta$ satisfying that  for all $\xi<\zeta$ we have $q_\zeta\rest\xi=q_\xi$. This definition is functional, so once we have
checked that it is correct, we can define $h_\zeta( \langle p_i:\,i<i^*\rangle)=q_\zeta$.

We have to check that the support of such $q_\zeta$ is still of size $<\mu$, so let $j^\ast=\cf(\zeta)$, hence $j^\ast<\mu$, and let us fix
an increasing sequence $\langle \zeta_j:\,j<j^\ast\rangle$ with $\zeta=\sup_{j<j^\ast}\zeta_j$. By the induction hypothesis we have that
$\bigcup_{\xi<\zeta}\supt(q_\xi)=\bigcup_{j<j^\ast}\supt(q_{\zeta_j})$, so is of size $<\mu$ by the regularity of $\mu$.
In particular $q_\zeta$ is indeed a condition in 
${\mathbb P}_\zeta$ and, clearly, for $\xi<\zeta$ we have $q_\zeta\rest\xi=q_\xi$. We have that $q_\zeta$ is a common upper bound
for $\langle p_i:\,i<i^*\rangle$. Moreover, the requirement $\supt(q_\zeta)=\bigcup_{i<i^\ast}\supt(p_i)$ is maintained.

The case \underline{$\zeta$ is a limit ordinal of cofinality $\ge\mu$} follows by defining $h_\zeta(\langle p_i:\,i<i^\ast\rangle)$ as follows.
First we find the first $\xi<\zeta$ such that $\bigcup_{i<i^\ast}\supt(p_i)\subseteq \xi$. Such $\xi$ exists since $\cf(\zeta)\ge\mu>i^\ast$.
Then we let 
\[
q_\zeta=h_\zeta(\langle p_i:\,i<i^\ast\rangle)=h_\xi(\langle p_i\rest\xi:\,i<i^\ast\rangle)\!\frown\![\emptyset_{\name{Q}_\varepsilon})_{\varepsilon\in [\xi,\zeta)},
\]
where $h_\xi(\langle p_i\rest\xi:\,i<i^\ast\rangle)$ is well-defined by the induction hypothesis. It also follows that $q_\zeta\rest \varepsilon=q_\varepsilon$ for
every $\varepsilon<\zeta$ and the requirement on the supports is maintained, as required.
$\eop_{\ref{lem:strong-closure}}$
\end{proof}

Lemma \ref{lem:strong-closure} isolated a property that will be important in the sequel, so we give it a name.

\begin{definition}\label{def:continously-strongly closed} We say that the limit ${\mathbb P}_{\zeta^\ast}$ of an iteration of forcing is 
{\em continuously strongly $(<\lambda)$-closed} if it is strongly $(<\lambda)$-closed and this can be witnessed by a function $h$
satisfying that for every $i^\ast<\lambda$ and an increasing  sequence $\bar{p}=\langle p_i:\,i<i^\ast\rangle$ of conditions in 
${\mathbb P}_{\zeta^\ast}$, we have $\supt(h(\bar{p}))=\bigcup_{i<i^\ast}\supt(p_i)$.
$\eod_{\mbox{\tiny{Def.}} \ref{def:continously-strongly closed}}$
\end{definition}

Hence, Lemma \ref{lem:strong-closure} proved that an iteration of $(<\lambda)$-strongly forcing with supports of size $<\mu$ for any regular 
$\mu\ge\lambda$ is continuously strongly $(<\lambda)$-closed.

\begin{theorem}\label{th:iteration} Suppose that $\langle ({\mathbb P}_\zeta, \name{Q}_\zeta):\,\zeta<\zeta^\ast\rangle$ is an iteration made with supports of size $(<\lambda)$ of individual forcings $\name{Q}_\zeta$ which are:
\begin{itemize}
\item stationary $\lambda^{+}$-cc and
\item strongly $(<\lambda)$-closed.

\end{itemize}

Then the limit ${\mathbb P}_{\zeta^\ast}$ of the iteration satisfies:
\begin{enumerate}
\item ${\mathbb P}_{\zeta^\ast}$ is continuously strongly $(<\lambda)$-closed, as witnessed by some function $h_{\zeta^\ast}$,
\item for every $\alpha<\lambda^+$,
the set 
\[
\DD^{\zeta^\ast}_\alpha=\{p\in {\mathbb P}_{\zeta^\ast}:\,(\forall \zeta\in \supt(p))[p\rest \zeta\mbox{ decides the value of }\name{f}_\zeta(\alpha) \mbox{ if defined}]\}
\]
is
dense in ${\mathbb P}_{\zeta^\ast}$,
\item  
${\mathbb P}_{\zeta^\ast}$ satisfies stationary $\lambda^{+}$-cc.
\end{enumerate}
\end{theorem}

\begin{proof}(of Theorem  \ref{th:iteration}) The proof if by induction on $\zeta^\ast$. The part (1) is  proven in Lemma \ref{lem:strong-closure}, where we need to take 
$\mu=\lambda$. Throughout the proof, we shall be taking advantage of the displayed equation in the statement of Lemma \ref{lem:strong-closure}.

We prove the items (2)-(3) simultaneously. Note that the sets $\DD^{\zeta^\ast}_\alpha$ are not open even though they are dense, since increasing a condition
in $\DD^{\zeta^\ast}_\alpha$ will in general change its support and potentially take it out of $\DD^{\zeta^\ast}_\alpha$. This point, as we shall see, will be handled by the fact that we have continuous strong $(<\lambda)$-closure. 

\medskip

{\em Case 1}. \underline{$\zeta^\ast=0$}. This case is trivial.

\medskip

{\em Case 2}.  \underline{$\zeta^\ast=\zeta+1$ for some $\zeta$}. We hence have that $\mathbb P_{\zeta^\ast}=\mathbb P_{\zeta}\ast \name{Q}_\zeta$. 
For item (2), given $(p, \name{q})\in \mathbb P_{\zeta}\ast \name{Q}_\zeta$ and $\alpha<\lambda^+$. By the induction hypothesis, we can find $r_0\in \DD^\zeta_\alpha$ with $r_0\ge p$. As in (1) of Lemma \ref{lem:two-step}, we find $r_1\ge r_0$ such that $r\in \mathbb P_{\zeta}$ and
$r$ decides the value of $\name{f}_\zeta(\alpha)$ if defined. We cannot conclude that $(r_1, \name{q})$ is in $\DD^{\zeta^\ast}_\alpha$, as we would be allowed to do if $\DD^{\zeta^\ast}_\alpha$ were open. However, we can extend $r_1$ to $r_2\in \DD^\zeta_\alpha$, by the induction hypothesis and then 
$(r_2, \name{q})$ is an extension of $(p, \name{q})$ which is in $\DD^{\zeta^\ast}_\alpha$.

We have already proven (3) in Lemma \ref{lem:two-step}. The only difference here is that we choose each $r_\alpha\in \DD^{\zeta^\ast}_\alpha$ first,
which is the reason we need (2).  

\medskip

{\em Case 3}. \underline{$\zeta^\ast>0$ is a limit ordinal and $j^\ast=\cf(\zeta^\ast)<\lambda$}. Let $\langle \zeta_j:\,j<j^\ast\rangle$ be a continuous increasing sequence of ordinals with $\sup_{j<j^\ast}\zeta_j=\zeta^\ast$ and $\zeta_0=0$.  We first prove (2). Given $\alpha<\lambda^+$ and a condition $p$. By induction on $j<j^\ast$ we shall choose $\langle r_j:j<j^\ast\rangle$ so that:
\begin{itemize}
\item $r_j\in \DD^{\zeta_j}_\alpha$,
\item $r_{j+1}\ge r_j\frown p\rest [\zeta_j, \zeta_{j+1})$,
\item for $j$ a limit ordinal $r_j=h_{\zeta_j}(\langle r_k \frown p\rest [\zeta_k, \zeta_{k+1}):\,k<j\rangle)$.
\end{itemize}
The induction is quite straightforward. Namely, the initial step is trivial since $\zeta_0=0$. Arriving at $j+1$ we first note that $r_j\frown p\rest [\zeta_j, \zeta_{j+1})$ is a condition in ${\mathbb P}_{\zeta_{j+1}}$, since $r_j \ge p\rest  \zeta_j$. Then by the induction hypothesis (2) applied to
$\zeta_{j+1}$ we can find 
$r_{j+1}\ge r_j\frown p\rest [\zeta_j, \zeta_{j+1})$ with $r_{j+1}\in \DD^{\zeta_{j+1}}_\alpha$. Note that we have possibly increased the support of 
$r_j$. Arriving at a limit ordinal $j$, we have by the choice of the
conditions $r_k$ for $k<j$ that $\langle r_k \frown p\rest [\zeta_k, \zeta_{k+1}):\,k<j\rangle$ is an increasing sequence in ${\mathbb P}_{\zeta_j}$.
Since $j<\lambda$, the upper bound $h_{\zeta_j}(\langle r_k \frown p\rest [\zeta_k, \zeta_{k+1}):\,k<j\rangle)$ is well defined, so $r_j$ is a 
well defined condition. By the fact that $h$ witnesses the continuous strong $(<\lambda)$-closure of ${\mathbb P}_{\zeta_j}$, we have that
$\supt(r_j)=\bigcup_{k<j}\supt(r_k)$. Since for every $k<j$ we have that $r_k\in \DD^{\zeta_k}_\alpha$, it follows 
by the definition of $\DD^{\zeta_j}_\alpha$ that $r_j\in \DD^{\zeta_j}_\alpha$. This finishes the proof of (3).

For (4), suppose that $\langle p_\alpha:\,\alpha<\lambda^{+}\rangle$
is a given sequence of conditions in ${\mathbb P}_{\zeta^\ast}$. By (3), without loss of generality, every $p_\alpha$ is an element of
$\DD^{\zeta^\ast}_\alpha$. By the induction hypothesis every ${\mathbb P}_{\zeta_j}$ is stationary $\lambda^{+}$-cc. In particular for every $j<j^\ast$, there is a club $C_j\subseteq \lambda^{+}$ in $\mathbf V$ and a regressive function $f_j$ on $C_j$ witnessing this for $\langle p_\alpha\rest \zeta_j:\,\alpha<\lambda^{+}\rangle$. Let $D$ and $g$ as provided by Fact \ref{lem:combining}, using $\langle f_j:\,j<j^\ast\rangle$ as the input.
Suppose now that $\alpha, \beta\in D\cap S^{\lambda^+}_\lambda$ are such that 
$g(\alpha)=g(\beta)$.
In particular, for every $j<j^\ast$, the conditions $p_\alpha\rest \zeta_j , q_\alpha\rest\zeta_j$ are compatible. By (1) applied to each $\zeta_j$, using the technique of the proof of Lemma \ref{lem:strong-closure} with $i^\ast=2$, we have
a concretely defined upper bound $r_j$ of $p_\alpha\rest \zeta_j , q_\alpha\rest\zeta_j$. Moreover, this definition is such that $j<j'\implies r_j'\rest\zeta_j=r_j$.
Then the required upper bound $r$ for $p_\alpha$ and $p_\beta$ is the unique condition $r\in {\mathbb P}_{\zeta^\ast}$ such that for every $j<j^\ast$
we have  $r\rest\zeta_j=r_j$.

\medskip

{\em Case 4}. \underline{$\zeta^\ast>0$ is a limit ordinal and $\cf(\zeta^\ast)=\lambda$}. 
Let $\langle \zeta_j:\,j<\lambda\rangle$ be a continuous increasing sequence of ordinals with $\sup_{j<\lambda}\zeta_j=\zeta^\ast$ and $\zeta_0=0$.
For (3) let $\alpha$ and $p$ be given. Note that there is $j$ such that $\supt(p)\subseteq \zeta_j$. By the induction hypothesis there is $r\in 
\DD^{\zeta_j}_\alpha$ with $r\ge p\rest \zeta_j$. Then $r\!\frown\!p\rest [\zeta_j, \zeta^\ast)$ is an extension of $p$ and an element of 
$\DD^{\zeta^\ast}_\alpha$.

Now we prove (4).
Suppose that $\langle p_\alpha:\,\alpha<\lambda^{+}\rangle$ in  ${\mathbb P}_{\zeta^\ast}$ is given. As before, we may extend each $p_\alpha$ to be in
$\DD^{\zeta^\ast}_\alpha$. In particular, we may assume that the stationary $\lambda^{+}$-cc of each ${\mathbb P}_{\zeta_j}$ with respect to 
$\langle p_\alpha\rest\zeta_j:\,\alpha<\lambda^{+}\rangle$  is
witnessed by a club $E_j$ of $\lambda^+$ and a regressive function $g_j$ defined on $E_j$. Let $E=\bigcap_{j<\lambda}E_j$.
Define a function $c:\,\lambda^+\to \lambda$
by 
\[
c(\alpha)=\min\{j<\lambda:\,\supt(p_\alpha)\subseteq \zeta_j\}.
\]
Hence $c$ is a regressive function on $\lambda^+\setminus \lambda$. Let $C=E\setminus \lambda$ and define $h$ on $C$ by $h(\alpha)=g_{c(\alpha)}(\alpha)$.
Note that $h$ is regressive.
By Fact \ref{lem:combining}, there is a club $D\subseteq C$ and a regressive function $f$ on $D$ such that for every $\alpha, \beta\in D\cap S^{\lambda¨+}_\lambda$
we have
\[
f(\alpha)=f(\beta) \implies [c(\alpha)=c(\beta) \wedge h(\alpha)=h(\beta)].
\]
We claim that for such $\alpha, \beta$ the conditions $p_\alpha, p_\beta$ are compatible. Indeed, let $j=c(\alpha)=c(\beta)$. Then
$\supt(p_\alpha), \supt(p_\beta)\subseteq \zeta_j$ and $p_\alpha\rest \zeta_j, p_\beta\rest \zeta_j$ are compatible by $g_j(p_\alpha))=g_j(p_\beta)$.
So $p_\alpha, p_\beta$ are compatible.

\medskip

{\em Case 5}. \underline{$\zeta^\ast>0$ is a limit ordinal and $\cf(\zeta^\ast)=\lambda^+$}. Let $\langle \zeta_j:\,j<\lambda\rangle$ be a continuous increasing sequence of ordinals with $\sup_{j<\lambda}\zeta_j=\zeta^\ast$ and $\zeta_0=0$.
The proof of (3) is the same as in the previous case.

For the proof of (4), let us use the same notation for
$\langle p_\alpha:\,\alpha<\lambda^+\rangle$, $E_j$'s and $g_j$'s. Let us define $E=\Delta_{j<\lambda^+}E_j$. 
We keep the same formula for $c$,
but it is now a function from $\lambda^+$ to $\lambda^+$. We observe that $|S=\bigcup_{\alpha<\lambda^+} \supt(p_\alpha)|\le\lambda^{+}$. If $|S|<\lambda^{+}$,
then there is $\zeta<\zeta^\ast$ such that for all $\alpha<\lambda^{+}$ we have $\supt(p_\alpha)\subseteq \zeta$ and hence the desired conclusion (4) follows from the induction hypothesis.  So we may assume that $|S|=\lambda^{+}$. Let
$S=\{\varepsilon_\alpha:\,\alpha<\lambda^{+}\}$ be the increasing enumeration. Let 
$t:\,\lambda^+\to\lambda$ be given by 
\[
t(\alpha)=\sup(\supt(p_\alpha)\cap \varepsilon_\alpha)).
\]
Therefore, $t$ is a function which is regressive on $S^{\lambda^{+}}_\lambda$. 

There is a club $C$ of $\lambda^+$ such that for every $\alpha\in S^{\lambda^+}_\lambda$.
for every
$\gamma<\alpha$ we have 
\begin{itemize}
\item $c(\gamma)<\varepsilon_\alpha$ and
\item $\varepsilon_\gamma<\alpha$.
\end{itemize}
Without loss of generality, $C\subseteq E$. Define a function $h(\alpha)=g_{t(\alpha)}(\alpha)$, hence $h$ is a regressive function on $C$. By applying
Fact \ref{lem:combining}, there is function $f$ regressive on  a club $D\subseteq C$ such that for every $\alpha, \beta\in D$, if $f(\alpha)=f(\beta)$, then
$t(\alpha)=t(\beta)$ and $h(\alpha)=h(\beta)$.

Suppose that $\alpha<\beta$ are both in $S^{\lambda^+}_\lambda \cap D$ and $f(\alpha)=f(\beta)$. Hence
$t=t(\alpha)=t(\beta)$ and $g_t(\alpha)= g_t(\beta)$.
Then $j=c(\alpha)<\varepsilon_\beta$ and $p_\alpha, p_\beta$ are compatible iff $p_\alpha\rest \zeta_j$ and $p_\beta\rest \zeta_j$ are compatible. Moreover, 
\[
\sup(\supt(p_\beta)\cap \zeta_j )\le \sup(\supt(p_\beta)\cap \varepsilon_\beta)=t=\sup(\supt(p_\alpha)\cap \varepsilon_\alpha)\le\varepsilon_\alpha<\beta.
\]
In particular, $\beta\in E_t$. We have that $t=t(\alpha)<\alpha$, so $\alpha\in E_t$ as well. On the other hand, $p_\alpha\rest \zeta_j, p_\beta\rest \zeta_j$ are compatible
iff $p_\alpha\rest t$ and $p_\beta\rest t$ are compatible, since $\supt(p_\beta)\cap [t,\zeta_j )=\emptyset$. But $p_\alpha\rest t$ and $p_\beta\rest t$ are compatible since $g_t(\alpha)=g_t(\beta)$. This shows that $D$ and $h$ witness that ${\mathbb P}_{\zeta^\ast}$ is $\lambda^+$-stationary-cc for
$\langle p_\alpha:\,\alpha<\lambda^+\rangle$, and since $\langle p_\alpha:\,\alpha<\lambda^+\rangle$ was arbitrary, ${\mathbb P}_{\zeta^\ast}$ is $\lambda^+$-stationary-cc.

\medskip
{\em Case 6}. \underline{$\zeta^\ast>0$ is a limit ordinal and $\cf(\zeta^\ast)>\lambda^+$}. (3) has the same proof as in the Cases 3. and 4. For (4), this case follows by the induction hypothesis, since in this case for every sequence $\langle p_\alpha:\,\alpha<\lambda^+\rangle$ of conditions in 
${\mathbb P}_{\zeta^\ast}$ we have $\sup[\bigcup_{\alpha<\lambda^+}\supt(p(\alpha))]<\zeta^\ast$.
$\eop_{\ref{th:iteration}}$
\end{proof}

\section*{Acknowledgements} 
The author used the AI product GROK 3 by X-AI, which served as a research assistant for rough historical facts, finding references and {\LaTeX} questions. 
All mathematical developments are due entirely to the author unassisted by the AI.

\bibliography{../bibliomaster}
\bibliographystyle{plainurl}

\end{document}